\documentclass{amsart}

\usepackage{hyperref}
\theoremstyle{plain}
\newtheorem{thm}{Theorem}[section]
\newtheorem{prop}[thm]{Proposition}
\newtheorem{conj}[thm]{Conjecture}
\newtheorem{que}[thm]{Question}
\newtheorem{prb}[thm]{Problem}
\theoremstyle{definition}
\newtheorem{re}[thm]{Remark}

\newcommand{\CC}{{\mathbb C}}
\newcommand{\RR}{{\mathbb R}}
\newcommand{\NN}{{\mathbb N}}
\def\sgn{{\rm sgn}\,}
\newcommand{\cha}{\operatorname{char}}
\newcommand{\Wedge}{\bigwedge\nolimits}

\newcommand{\ELS}{\mathrm{ELS}}
\newcommand{\OLS}{\mathrm{OLS}}
\newcommand{\RBC}{\mathrm{RBC}}
\newcommand{\OBC}{\mathrm{OBC}}
\newcommand{\ATLSC}{\mathrm{ATLSC}}
\newcommand{\dett}{\mathrm{det}_{n-1}}

\begin{document}

\title{An online version of Rota's basis conjecture}
\author{Guus P.~Bollen and Jan Draisma}
\address{
Department of Mathematics and Computer Science\\
Technische Universiteit Eindhoven\\
P.O. Box 513, 5600 MB Eindhoven, The Netherlands}
\thanks{JD is supported by a Vidi grant from the Netherlands
Organisation for Scientific Research (NWO)}
\email{g.p.bollen@tue.nl, j.draisma@tue.nl}

\begin{abstract}
Rota's basis conjecture states that in any square array of vectors whose
rows are bases of a fixed vector space the vectors can be rearranged
within their rows in such a way that afterwards not only the rows are
bases, but also the columns. We discuss an {\em online version} of this
conjecture, in which the permutation used for rearranging the vectors in
a given row must be determined without knowledge of the vectors further
down the array. The paper contains surprises both for those who believe
this {\em online basis conjecture} at first glance, and for
those who disbelieve it.
\end{abstract}

\maketitle

\section{Background}

Fix a field $K$ and $n \in \NN$. The following conjecture
first appeared in \cite{Huang94}.

\begin{conj}[Rota's basis conjecture, RBC$_n(K)$]
Let $V=K^n$ and let $v=(v_{ij})_{ij}$
be an $n \times n$-array of vectors in $V$. Suppose that for all
$i=1,\ldots,n$ the $i$-th row $(v_{i,1},\ldots,v_{i,n})$
of $v$ is a basis of $V$. Then there exist permutations
$\pi_1,\ldots,\pi_n \in S_n$ such that for all $j=1,\ldots,n$ the vectors
$v_{1,\pi_1(j)},\ldots,v_{n,\pi_n(j)}$ form a basis, as well.
\end{conj}

Here the vectors $v_{1,\pi_1(j)},\ldots,v_{n,\pi_n(j)}$ form the $j$-th
{\em column} of the array obtained from $v$ by applying, for each
$i=1,\ldots,n$, the $i$-th permutation $\pi_i$ to the $i$-th row. $\RBC$
has a natural generalisation to arbitrary matroids, and there is quite
some literature on this generalisation: e.g., it is true in rank up to
three \cite{Chan95}, for paving matroids \cite{Geelen06}, and a
version for $k \times n$-arrays with $n>\binom{k+1}{2}$ was proved in
\cite{Geelen07}. In this paper, however, we restrict ourselves to the
vector space version except for a brief remark in the last
section.

It is well known that $\RBC_n$ is related to a conjecture by Alon and
Tarsi on Latin squares. Recall that a {\em Latin square of order
$n$} is an $n \times n$-array $a=(a_{ij})_{ij}$ with all $a_{ij} \in
[n]:=\{1,\ldots,n\}$ such that every element of $[n]$ occurs precisely once in
each row and each column. The map $j \mapsto a_{ij}$ is then an element of
the symmetric group $S_n$ for each fixed $i$, and so is the map $i \mapsto
a_{ij}$ for each fixed $j$, and the {\em sign} of $a$ is by definition
the product of the signs of these $2n$ permutations. A Latin square is
called {\em even} if its sign is $1$ and {\em odd} if its sign is $-1$.
The number of even Latin squares is denoted $\ELS(n)$, and the number
of odd Latin squares is denoted $\OLS(n)$. If $n$ is odd and $>2$,
then interchanging rows $1$ and $2$ gives a sign-reversing involution
on the set of Latin squares, which proves that $\ELS(n)=\OLS(n)$.
The following conjecture from \cite{Alon92} on the case of even $n$
has various reformulations \cite{Janssen95}.

\begin{conj}[Alon and Tarsi's Latin square conjecture, ATLSC$_n$]
If $n$ is even, then $\ELS(n) \neq \OLS(n)$.
\end{conj}

The known relation between RBC$_n(K)$ and ATLSC$_n$ is the following,
and has been proved by various authors \cite{Huang94,Onn97,Wild94}.

\begin{thm}\label{thm:ATLSCRBC}
If the characteristic of $K$ does not divide $\ELS(n)-\OLS(n)$, then
$\RBC_n(K)$ holds. In particular, if the characteristic of $K$ is zero
or sufficiently large, then $\ATLSC_n$ implies $\RBC_n(K)$.
\end{thm}

$\ATLSC_n$ has been proved for $n$ of the form $p+1$ \cite{Drisko97}
and for $n$ of the form $p-1$ \cite{Glynn10}, both with $p$ an odd
prime (see also \cite{Berndsen12} for easy, self-contained proofs).
The paper \cite{Zappa97} contains a claim to the effect that $\ATLSC_m$
implies $\ATLSC_{2m}$. As pointed out in \cite{Glynn10}, the proof is
incorrect. However, in Section~\ref{sec:Even} we show that $\ATLSC_m$
does imply (an online strengthening of) a variant of $\RBC$ for $(m+1)
\times n$-arrays of vectors with $n>m$ a multiple of $m$. 

There is a generalisation of $\ATLSC$ in which one counts only the
(signs of the) Latin squares whose diagonal entries are fixed to be, say,
$n$ \cite{Zappa97}.  While for even $n$ this generalised conjecture is
equivalent to $\ATLSC_n$, the generalised conjecture is also meaningful
for odd $n$, and was proved in \cite{Drisko98} for $n=p$ prime.
In Section~\ref{sec:Odd} we establish a connection to a version of
$\RBC_n$ where all bases share a common vector. This is different from
the result of \cite{Aharoni11}, which links a variant of $\ATLSC$ to a
different weakening of $\RBC$.

\section{Online version, main results}

This paper concerns the following natural strengthening of
$\RBC_n(K)$.

\begin{que}[Online version of Rota's basis conjecture, $\OBC_n(K)$]
Does there exist an algorithm for finding the permutations $\pi_i$
in $\RBC_n(K)$ which is {\em online} in the following sense? The rows
$(v_{1j})_j,(v_{2j})_j,\ldots,(v_{nj})_j$ of $v$ are given sequentially
to the algorithm, and directly after reading the $i$-th row the algorithm
fixes permutation $\pi_i$, without knowledge of the remaining $n-i$ rows.
\end{que}

For maximum effect, we suggest that the reader take a few minutes
to determine his own position on this conjecture. Here are some
considerations s/he might take into account: certainly $\pi_1$ can be
fixed in any manner, and then the second row can be adapted to the first,
so $\OBC_2(K)$ is equivalent to $\RBC_2(K)$---and true, of course. Also,
the online problem is reminiscent of a very easy instance of the problem
of completing a partially filled Latin square: indeed, if all the bases
given to the algorithm are equal to the standard basis $e_1,\ldots,e_n$
of $V=K^n$, then it does not matter how $\pi_1,\ldots,\pi_i$ are chosen,
as long as they do not put twice the same standard basis vector in
any column. The $i \times n$ rectangle thus obtained can always be
completed to a Latin square; one way to see this is to use the known
fact that a regular bipartite graph with non-empty set of edges always
contains a perfect matching. On the negative side, with more general
vectors than just standard basis vectors one can run into
trouble, e.g. as
follows. In this $2 \times 4$-array of vectors in $V$: \begin{align*}
v_{11}&=e_1 & v_{12}&=e_2 & v_{13}&=e_3 & v_{14}&=e_4\\ v_{21}&=e_2 &
v_{22}&=e_1+e_2 & v_{23}&=e_1+e_3 & v_{24}&=e_1+e_4 \end{align*} the
two rows are bases, and the four {\em column spaces} \begin{align*}
V_1&:=\langle e_1,e_2 \rangle, &V_2&:=\langle e_2,e_1+e_2\rangle,&
V_3&:=\langle e_3,e_1+e_3 \rangle,& V_4&:=\langle e_4,e_1+e_4 \rangle
\end{align*} span two-dimensional spaces, which however all intersect
in $\langle e_1 \rangle$. Hence any online algorithm would have to avoid
this arrangement, because the next, third basis might contain the vector
$e_1$. More generally, an online algorithm should avoid any $i \times
n$-arrangement where the intersection of $\ell$ of the $i$-dimensional
column spaces $V_1,\ldots,V_n$ constructed so far has dimension strictly
greater than $n-\ell$. Indeed, if $U:=V_1 \cap \cdots \cap V_\ell$ has
dimension $>n-\ell$, then an $(i+1)$-st basis of $V$ containing $>n-\ell$
elements of $U$ cannot be matched with $V_1,\ldots,V_n$. Conversely, if
no $\ell$-intersection has dimension $>n-\ell$, Hall's marriage theorem
ensures that any $(i+1)$-st basis can be matched. One might try and
make these ``general position'' conditions more strict and dependent on
$i$, so that given any tuple of $i$-dimensional spaces $V_1,\ldots,V_n$
satisfying the conditions for $i$ and given any basis $v_1,\ldots,v_n$,
a permutation $\pi$ exists such that $V_1 + \langle v_{\pi(1)} \rangle,
\ldots, V_n + \langle v_{\pi(n)} \rangle$ are all $(i+1)$-dimensional and
satisfy the next set of conditions. But very quickly these conditions
seem to grow rather intricate, and in Section~\ref{sec:Open} we will
see that in fact they cannot be formulated purely in matroidal terms.

By now the reader may have formed his own opinion about the answer to
$\OBC_n$. Ready for our results? Here we go!

\begin{thm} \label{thm:Even}
If the characteristic of $K$ does not divide $\ELS(n)-\OLS(n)$, then
$\OBC_n(K)$ holds.
\end{thm}

In particular, in the cases where $\RBC_n(K)$ holds as a consequence
of $\ATLSC_n$ (hence with $n$ even), in fact the stronger statement
$\OBC_n(K)$ holds.

\begin{thm} \label{thm:Odd}
For any odd $n>2$ and any field $K$ that contains a primitive
$m$-th root of unity for all odd $m \leq n$, $\OBC_n(K)$ is false.
\end{thm}

Consequently, while $\OBC_n(\CC)$ is true for infinitely many and
conjecturally all {\em even} $n$, it is false for all {\em odd} $n$
except the trivial case $n=1$.  We were quite surprised by this dichotomy
between even $n$ and odd $n$. In the subsequent two sections of this
paper we prove Theorems~\ref{thm:Even} and~\ref{thm:Odd}, respectively,
and discuss some variants for other shapes of arrays. We conclude with
some further remarks in Section~\ref{sec:Open}.

\section{Even dimensions} \label{sec:Even}

In this section we prove Theorem~\ref{thm:Even}. This implies
Theorem~\ref{thm:ATLSCRBC}, and in fact our contribution may be
interpreted as giving more ``semantics'' to existing proofs of the
latter. Since the validity of $\OBC_n(K)$ for a field $K$ implies
its validity over any subfield of $K$, we assume that $K$ is
algebraically closed. We will use algebro-geometric terminology such
as {\em hypersurface} for a variety in some vector space over $K$
defined by the vanishing of a single non-constant
polynomial. We write $\Wedge^k V$
for $k$-th exterior power of $V$, which is the quotient of the $k$-fold
tensor power $V^{\otimes k}$ by the subspace spanned by all pure tensors
$v_1 \otimes \cdots \otimes v_k$ that have $v_i=v_j$ for some distinct
indices $i,j$. The image in $\Wedge^k V$ of $v_1 \otimes \cdots \otimes
v_k$ for arbitrary vectors $v_1,\ldots,v_k$ is denoted by $v_1 \wedge
\cdots \wedge v_k$. It is non-zero if and only if the $v_i$ are linearly
independent, and then for any other basis $w_1,\ldots,w_k$ of $\langle
v_1,\ldots,v_k \rangle$ the element $w_1 \wedge \cdots \wedge w_k$ is
a scalar multiple of $v_1 \wedge \cdots \wedge v_k$.  In this sense,
(images of) pure tensors modulo scalar factors bijectively represent
$k$-dimensional subspaces of $V$.

\begin{proof}[Proof of Theorem~\ref{thm:Even}.]
For $k=1,\ldots,n$ and $\pi \in S_n$ consider the morphism
of algebraic varieties
\[ \Psi_{k,\pi}:\left(\Wedge^{k-1} V\right)^n \times V^n \to
\left(\Wedge^k V\right)^n,\
	((\omega_j)_j,(u_j)_j) \mapsto (\omega_j \wedge
	u_{\pi(j)})_j,
\]
which models adding the $k$-th row $(u_1,\ldots,u_n)=(v_{kj})_j$ of the
array to the $n$-tuple of $(k-1)$-dimensional spaces constructed so far,
according to the permutation $\pi$.

Let $D \subseteq V^n$ denote the hypersurface, defined by the determinant
$\det \in (V^*)^{\otimes n}$, consisting of all $n$-tuples of vectors that
do {\em not} form a basis.  Assuming that the characteristic of $K$ does
not divide $\ELS(n)-\OLS(n)$, we will prove that for each $k=0,\ldots,n$
there exists a hypersurface $H_k$ in $(\Wedge^k V)^n$, the Cartesian
product of $n$ copies of the $k$-th exterior power of $V$, having the
following two properties:
\begin{enumerate}
\item each $H_k$ contains the set of tuples
$(\omega_1,\ldots,\omega_n) \in \left(\Wedge^k V\right)^n$
for which some $\omega_i$ is zero; and
\item $H_{k-1}$ and $H_k$ are related by \[ \bigcap_{\pi \in S_n}
(\Psi_{k,\pi}^{-1} H_k) \subseteq \left( H_{k-1} \times V^n \right) \cup
\left(\left(\Wedge^{k-1} V\right)^n \times D\right). \]
In other words, if for some $(\omega,u) \in (\Wedge^{k-1}
V)^n \times V^n$ it is true that for all $\pi \in S_n$,
$\Psi_{k,\pi}(\omega,u)$ is in $H_k$, then either $\omega$ is already in $H_{k-1}$ or $u$ is a linearly dependent system (or both).
\end{enumerate}
The latter condition implies that if $(\omega_1,\ldots,\omega_n)
\in (\Wedge^{k-1} V)^n$ lies outside $H_{k-1}$, then for any basis
$u_1,\ldots,u_n$ of $V$ there exists a permutation $\pi$
such that the tuple
$(\omega_1 \wedge u_{\pi(1)},\ldots,\omega_n \wedge
u_{\pi(n)})$ lies outside $H_k$.

The online algorithm is then as follows. Take any point
$(\omega_{0,1},\ldots,\omega_{0,n}) \in (\Wedge^0 V)^n$ outside
$H_0$. After having constructed $(\omega_{k-1,1},\ldots,\omega_{k-1,n}) \in
(\Wedge^{k-1} V)^n \setminus H_{k-1}$ for $k \geq 1$ read the $k$-th
row $(v_{k1},\ldots,v_{kn})$ of the array and choose $\pi_{k}$ such
that $(\omega_{kj})_j:=\Psi_{k,\pi_k}((\omega_j)_j,(v_{kj})_j) \not \in
H_k$. This yields permutations $\pi_1,\ldots,\pi_n$ and an 
$n$-tuple of pure tensors $(\omega_{n,1},\ldots,\omega_{n,n})$ that are
all non-zero since the tuple lies outside $H_n$, and we are
done.

It remains to construct the $H_k$. For this it is convenient
to work with ordinary tensors first and then take the quotient to
arrive at exterior powers. Thus consider
\[ \Phi_{k,\pi}: (V^{\otimes k-1})^n \times V^n \to
(V^{\otimes k})^n,\ ((\omega_j)_j,(v_j)_j) \mapsto (\omega_j \otimes
v_{\pi(j)})_j \]
and its co-morphism, regarded as a map
\[ \Phi_{k,\pi}^*:
	((V^*)^{\otimes k})^{\otimes n} \to
	((V^*)^{\otimes k-1})^{\otimes n} \otimes
	(V^*)^{\otimes n}.
\]
More explicitly, for numbers
$a_{11},\ldots,a_{kn} \in [n]$ we use the notation
\begin{equation} \label{eq:Tensor}
	\begin{bmatrix}
		a_{11} & \ldots & a_{1n}\\
		\vdots & & \vdots \\
		a_{k1} & \ldots & a_{kn}
	\end{bmatrix}
\end{equation}
for the tensor product $(x_{a_{11}} \otimes \cdots \otimes x_{a_{k1}}) \otimes\cdots \otimes (x_{a_1n} \otimes \cdots \otimes x_{a_{kn}})$,
an element of $((V^*)^{\otimes k})^{\otimes n}$, where $x_1,\ldots,x_n
\in V^*$ are the coordinate functions on $V=K^n$ dual to the
standard basis $e_1,\ldots,e_n$. The linear map $\Phi_{k,\pi}^*$
sends \eqref{eq:Tensor} to
\[
	\begin{bmatrix}
		a_{11} & \ldots & a_{1n}\\
		\vdots & & \vdots \\
		a_{k-1,1} & \ldots & a_{k-1,n}
	\end{bmatrix}
	\otimes
	\begin{bmatrix}
	a_{k\pi^{-1}(1)} & \ldots &a_{k\pi^{-1}(n)}
	\end{bmatrix}.
\]
Hence, the map $\sum_{\pi} \sgn(\pi) \Phi_{k,\pi}^*$ maps
\eqref{eq:Tensor} to
\[ 
	c \cdot 
	\begin{bmatrix}
		a_{11} & \ldots & a_{1n}\\
		\vdots & & \vdots\\
		a_{k-1,1} & \ldots & a_{k-1,n}
	\end{bmatrix}
	\otimes \det,
	\text{ where }
\det=\sum_{\pi \in S_n} \sgn(\pi)
	\begin{bmatrix} \pi(1) & \ldots & \pi(n)
	\end{bmatrix}
\]
and 
where the scalar $c$ is $0$ if the numbers $a_{kj} \in [n],\ j=1,\ldots,n$
are not all distinct, and equal to the sign of the permutation $j
\mapsto a_{kj}$ otherwise. 

Let $\Theta_k:((V^*)^{\otimes k})^{\otimes
n} \to ((V^*)^{\otimes k-1})^{\otimes n}$ be the linear map that sends
\eqref{eq:Tensor} to 
\[ 
	c \cdot 
	\begin{bmatrix}
		a_{11} & \ldots & a_{1n}\\
		\vdots & & \vdots\\
		a_{k-1,1} & \ldots & a_{k-1,n}
	\end{bmatrix},
\]
without the determinantal factor.
The natural projection $V^{\otimes k} \to \Wedge^k V$ identifies,
dually, the space of linear functions on $\Wedge^k V$ with the space of
alternating elements of $(V^*)^{\otimes k}$, i.e., tensors
which change with $\sgn(\sigma)$ if a permutation $\sigma \in S_k$ is
applied to the tensor factors. We take $F_n:=\det^{\otimes
n}$ as defining equation for $H_n \subseteq (\Wedge^n V)^n$. Define
$F_{n-1},\ldots,F_0$ inductively by $F_{k-1}:=\Theta_k F_k$. Each $F_k$
is an $n$-linear function on $(\Wedge^k V)^n$, so that its zero set
$H_k$ contains all $n$-tuples for which some entry is zero. Moreover,
by construction a linear combination of the pull-backs $\Psi_{k,\pi}^*
F_k=\Phi_{k,\pi}^* F_k$ over all $\pi$ equals $F_{k-1} \otimes \det$. This
implies the second requirement on the $H_k$. The only thing that could
go wrong is that $F_k=0$, in which case $H_k$ is the entire space rather
than a hypersurface. Hence we need to verify that
\[ \Theta_1 \circ \Theta_2 \circ \cdots \circ \Theta_n F_n
\neq 0. \]
In $F_n$ only tensors of the form
\begin{equation}
	\begin{bmatrix}
		a_{11} & \ldots & a_{1n}\\
		\vdots & & \vdots \\
		a_{n1} & \ldots & a_{nn}
	\end{bmatrix}
\end{equation}
appear in which all columns are permutations of $[n]$, and
then with coefficient equal to the product of the signs of these
permutations. Moreover, $\Theta_1\cdots \Theta_n$ kills such a tensor
unless each of its rows is also a permutation of $[n]$. Thus only
Latin squares contribute to the expression above, and each Latin square
contributes exactly its sign. Hence we find that
\[ \Theta_1 \circ \cdots \circ \Theta_n F_n =
\ELS(n)-\OLS(n), \]
and this proves Theorem~\ref{thm:Even}. 
\end{proof}

\begin{re}
Of course $\Theta_n F_n$ is just some non-zero scalar multiple of a
suitable determinant. Next we claim that $\Theta_{n-1} \Theta_n F_n$ is
zero for odd $n$ and non-zero for even $n$. To prove both
statements note that
\begin{equation} \label{eq:Tensor3}
	\begin{bmatrix}
		a_{11} & \ldots & a_{1n}\\
		\vdots & & \vdots \\
		a_{n-2,1} & \ldots & a_{n-2,n}
	\end{bmatrix}
\end{equation}
can have a non-zero coefficient in $\Theta_{n-1} \Theta_n F_n$ only if
it can be extended to a tensor
\[
	\begin{bmatrix}
		a_{11} & \ldots & a_{1n}\\
		\vdots & & \vdots \\
		a_{n1} & \ldots & a_{nn}
	\end{bmatrix}
\]
in which the last two rows {\em and} all columns are permutations
of $[n]$. The coefficient of \eqref{eq:Tensor3} in $\Theta_{n-1}
\Theta_n F_n$ is then the product of these $n+2$ signs, summed over all
possible extensions. For odd $n$ interchanging the last two rows gives a
sign-reversing involution on the set of extensions, hence $\Theta_{n-1}
\Theta_n F_n$ is zero, as claimed. For $n$ even consider, for instance,
the tensor
\[
	\begin{bmatrix}
		1 & 2 & \ldots & n\\
		2 & 3 & \ldots & 1\\
		\vdots & \vdots & & \vdots \\
		n-2 & n-1 & \ldots & n-3\\
	\end{bmatrix}.
\]
It has two possible extensions, with last two rows equal to
\[ \begin{bmatrix}
	n-1 & n & \ldots & n-2\\
	n & 1 & \ldots & n-1
\end{bmatrix} \text{ and }
\begin{bmatrix}
	n & 1 & \ldots & n-1\\
	n-1 & n & \ldots & n-2\
\end{bmatrix}, \]
respectively. The product of the $2$ row signs is the same for both
extensions, and the product of the $n$ column signs is multiplied by
$(-1)^n$, which is an even power. Hence the tensor above has a coefficient with
absolute value $2$ in $\Theta_{n-1} \Theta_n F_n$. This means that
$H_{n-2}$ is a proper hypersurface (provided that $\cha K
\neq 2$). With a bit more effort, using the
fact that the Grassmannian of $(n-2)$-subspaces of $V$ is cut out from
$\Wedge^{n-2} V$ by quadratic polynomials one can show that $H_{n-2}$
does not contain all $n$-tuples of pure $(n-2)$-tensors, so that a
sufficiently general $n$-tuple of $(n-2)$-spaces can be extended with the last
two rows. In particular, this readily implies that given a $3 \times
n$-array whose rows are bases of $K^n$ with $n$ even and $\geq 4$,
there are always three permutations leading to an array in which the
columns are independent triples. 
\end{re}

A similar argument shows the following.

\begin{prop}
Let $\ell, m$ be natural numbers with $\ell>1$ and set $n:=\ell \cdot m$.
If the characteristic of $K$ does not divide $\ELS(m)-\OLS(m)$, then
there exists an online algorithm that sequentially takes $m+1$ bases of
$V=K^n$ as input and immediately after reading the $i$-th basis arranges
it as the $i$-th row in an $(m+1) \times n$-array such that each column
of the array eventually consists of $m+1$ linearly independent vectors.
\end{prop}

\begin{proof}
As in the previous remark, it suffices to show that $F_{n-m}$ is
non-zero, and for this it suffices to exhibit 
a basis vector in $((V^*)^{\otimes (n-m)})^{\otimes n}$
that has a non-zero coefficient in $F_{n-m}$. For this, take any basis
vector of the form
\[ \begin{bmatrix} A_1 & A_2 & \ldots & A_\ell \end{bmatrix}
\]
where all columns of $A_j$ contain exactly the $n-m$ numbers
$\{1,\ldots,n\} \setminus \{(j-1)m+1,\ldots,jm\}$. The coefficient in
$F_{n-m}$ is the sum, over all extensions to $n \times n$-arrays in which
all the last $m$ rows and all columns are permutations, of the products
of their $n+m$ signs. Up to a sign, this is just $(\ELS(m)-\OLS(m))^\ell$,
hence non-zero in $K$ by assumption.
\end{proof}

For some values of $n$ this proposition gives results sharper than those
in \cite{Geelen07}, but only for matroids representable over fields of
characteristic zero or sufficiently large.

\section{Odd dimensions} \label{sec:Odd}

In this section we prove Theorem~\ref{thm:Odd}. So we assume that $n$
is odd and that $K$ contains primitive $m$-th roots of unity for all
odd $m \leq n$, and we will show that no online algorithm exists for
arranging the basis vectors.

\begin{proof}[Proof of theorem~\ref{thm:Odd}.]
We argue that we can force any online algorithm for
choosing the row permutations into making an error. For this, we first
feed the algorithm $n-2$ times the standard basis $e_1,\ldots,e_n$.
If it does not make an error yet, then it will arrange these bases as
the rows of an $(n-2) \times n$-array in which each column consists of
$n-2$ distinct standard basis vectors.  For $l=1,\ldots,n$ let $V_l$
denote the space spanned by the vectors in the $l$-th column.

Now we construct a graph with vertex set $[n]$ and an edge
$\{i,j\}$ for
every $l$ with $e_i,e_j \not \in V_l$. This graph has $n$ edges, some
of which may be double, as some of the $V_l$ may coincide. Each $e_i$
is missing from exactly two of the $V_l$, so the graph is regular of degree
$2$ and hence a union of cycles. Since $n$ is odd, the graph has an
$m$-cycle for some odd $m \leq n$. Without loss of generality, we may
assume that for $l=1,\ldots,m-1$ the space $V_l$ misses exactly $e_l$ and
$e_{l+1}$ and that $V_m$ misses exactly $e_m$ and $e_1$.

We produce an $(n-1)$-st basis $v_1,\ldots,v_n$ for the algorithm to
arrange, as the columns of the following block matrix:
\[ \begin{bmatrix} Y & 0 \\ 0 & Z \end{bmatrix}, \]
where $Y$ has size $m \times m$ and $Z$ has size $(n-m)
\times (n-m)$, and where
\[ Y=\begin{bmatrix}
\zeta & \zeta^2 & \ldots & \zeta^m \\
\zeta^2 & \zeta^4 & \ldots & \zeta^{2m}\\
\vdots & \vdots & & \vdots\\
\zeta^m & \zeta^{2m} & \ldots & \zeta^{m^2}
\end{bmatrix}
\]
with $\zeta$ a primitive $m$-th root of unity, so that $\det
Y$ is non-zero. It does not matter what
$Z$ is, as long as it is also invertible, so that
$v_1,\ldots,v_n$ are legitimate input to the algorithm.

Since $V_{m+1},\ldots,V_n$ already contain the vectors
$v_1,\ldots,v_m$, the latter must be assigned
in some order to $V_1,\ldots,V_m$. Suppose that $V_l$ gets the
$\pi(l)$-th vector $v_{\pi(l)}$, where $\pi \in S_n$ stabilises the sets
$[m]$ and $\{m+1,\ldots,n\}$. Set $V_l':=V_{l} \oplus
\langle v_{\pi(l)} \rangle$,
where the sum is, indeed, direct if the algorithm makes no
error yet.

For $l=1,\ldots,m$, let $z_l \in V^*$ be a normal vector of the
space $V_l'$. Thus $z_l$ annihilates
all standard basis vectors except $e_l$ and $e_{l+1}$ (index
modulo $m$ with offset $1$), and is hence of the form $a x_l + b x_{l+1}$. The scalars $a$ and $b$ are determined by the fact that $z_l$ also
annihilates the vector with first $m$ entries $\zeta^{i \cdot \pi(l)},
i=1,\ldots,m$. Thus we find that $a \zeta^{l \cdot \pi(l)} + b \zeta^{(l+1) \cdot \pi(l)} = 0$, and we can choose
\[ z_l=\zeta^{\pi(l)} x_l - x_{l+1}. \]
We claim that $z_1,\ldots,z_m$ are linearly dependent. To
see this, we compute the determinant of their coefficients
with respect to $x_1,\ldots,x_m$:
\[ 
\det \begin{bmatrix}
     \zeta^{\pi(1)} & & & & -1\\
-1                  & \zeta^{\pi(2)} &        &                 & \\
&-1              & \ddots &                 & \\
&& \ddots &\zeta^{\pi(m-1)} & \\
&& &-1               & \zeta^{\pi(m)}
\end{bmatrix}
=\zeta^{\pi(1)+\pi(2)+\cdots+\pi(m)} - 1.
\] 
Here the minus sign is in fact independent of the parity of $m$: for $m$
odd, the $m$-cycle is an even permutation but the number of minus signs
is odd, while for $m$ even, the $m$-cycle is odd but the number of minus
signs is even. Now the exponent equals
\[
\pi(1)+\pi(2)+\cdots+\pi(m)=1+2+\cdots+m=\frac{1}{2}m(m+1). \]
Since $m$ is odd, the latter number is a multiple of $m$, and hence the determinant is zero.

Since $z_1,\ldots,z_m$ are linearly dependent, we find that $V_1' \cap
\cdots \cap V_m'$ has co-dimension strictly larger than $m$. Then $V_1'
\cap \cdots \cap V_n'$ has positive dimension. If we choose,
for the $n$-th basis, any basis containing a vector in that
intersection, the algorithm is forced to make an error.
This concludes the proof of Theorem~\ref{thm:Odd}.
\end{proof}

\begin{re}
For $n=3$, the above can be shown to be the only counterexample to
$\OBC_3$ up to symmetries, by a computation similar to the one in
\cite[Proof of Theorem 2]{Chow95}. So $\OBC_3(K)$ holds for fields $K$
not containing primitive cube roots of unity, such as $K=\RR$.  We did
not study the validity of $\OBC_n(\RR)$ for other odd values of $n$.
\end{re}

We conclude this section with a relation between the generalisation from
\cite{Zappa97} of $\ATLSC_n$ to odd $n$ and a special case of $\OBC_n$. A
different relation between a variant of $\ATLSC$ and a weakening of $\RBC$
for odd $n$ was established in \cite{Aharoni11}.

\begin{prop}
Consider only $n \times n$ Latin squares that have all
diagonal entries equal to the symbol $n$. If the number of even Latin squares minus the number of
odd Latin squares among these is non-zero when regarded as an element of
$K$, then there exists an online algorithm that sequentially takes $n$
bases of $V=K^n$, each of which contains the standard basis vector $e_n$,
and arranges these as the rows of an $n \times n$-array whose columns
are also bases.
\end{prop}

\begin{proof}
Write $W:=K^{n-1}$, so that $V=W \oplus \langle e_n \rangle$. The
algorithm will first replace each basis vector unequal to $e_n$ by its
projection in $W$. The proof is now identical to that of 
Theorem~\ref{thm:Even}, except that
in the $k$-th row we put the copy of $e_n$ on position $k$. Hence for
each $k \in [n]$ and every $\pi \in S_n$ with $\pi(k)=k$ we define
\begin{align*} &\Phi_{k,\pi}: (V^{\otimes k-1})^n \times
W^{[n]\setminus\{k\}} \to
(V^{\otimes k})^n,\\ 
&((\omega_j)_j,(w_j)_j) \mapsto 
(\omega_1 \otimes w_{\pi(1)},\ldots,\omega_k \otimes
e_n,\ldots,\omega_n \otimes w_{\pi(n)}). \end{align*}
Then $\Phi_{k,\pi}^*$ maps a tensor \eqref{eq:Tensor} to $0$
if $a_{kk} \neq n$ or if $a_{kj}=n$ for some $j \neq k$, 
and to 
\[
\begin{bmatrix}
a_{11} & \ldots & a_{1n}\\
\vdots & & \vdots \\
a_{k-1,1} & \ldots & a_{k-1,n}
\end{bmatrix}
\otimes
\begin{bmatrix}
a_{k\pi^{-1}(1)} & \ldots & \widehat{a_{k\pi^{-1}(k)}} & \ldots &a_{k\pi^{-1}(n)}
\end{bmatrix}
\]
(with $\hat{\cdot}$ indicating that that factor is left out) otherwise. 
Consequently, the linear combination $\sum_{\pi \in S_n: \pi(k)=k}
\sgn(\pi) \Phi^*$ maps a tensor \eqref{eq:Tensor} to
\[
	c \cdot 
	\begin{bmatrix}
		a_{11} & \ldots & a_{1n}\\
		\vdots & & \vdots\\
		a_{k-1,1} & \ldots & a_{k-1,n}
	\end{bmatrix}
	\otimes \dett
\]
where
\[ 
\dett=(-1)^{n-k} \sum_{\sigma \in S_{n-1}} \sgn(\sigma)
	\begin{bmatrix} \sigma(1) & \ldots & \sigma(n-1)
	\end{bmatrix}
\]
and where $c=\sgn(\tau)$ if the map $\tau: j \mapsto a_{kj}$ is a
permutation of $[n]$ with $\tau(k)=n$, and $c=0$ otherwise. The factor
$(-1)^{n-k}$ in $\dett$ accounts for the relation $\pi=(k,k+1,\ldots,n)
\tau$.  Let
\[ \Theta_k:((V^*)^{\otimes k})^{\otimes n} \to 
((V^*)^{\otimes k-1})^{\otimes n} \]
be the linear map taking \eqref{eq:Tensor} to the previous
expression without the factor $\det_{n-1}$. Starting with $F_n$ as in
Section~\ref{sec:Even} and defining $F_{k-1}:=\Theta_k F_k$, we find that
$F_0$ equals the sum of the signs of all Latin squares in the statement
of the proposition.
\end{proof}

\section{Further remarks} \label{sec:Open}

Both $\ATLSC$ and $\RBC$ are notoriously difficult conjectures, and our
results do not shed new light on these conjectures. But our results do
lead to a number of problems that might be more tractable than $\ATLSC$
and $\RBC$ themselves.

First, our construction in the previous section showed that for odd $n$
any online algorithm can be forced to make an error in the very last
row. Are there examples where the algorithm is forced to make an error
earlier? Or, on the contrary, can one formulate a combinatorial conjecture
\`a la $\ATLSC$ which, using algebro-geometric techniques, would imply the
existence of an online algorithm for arranging any $(n-1) \times n$-array?

Second, one can try to formulate an online version of Rota's basis
conjecture for general matroids, where the algorithm does not know 
the entire matroid on $n^2$ elements in advance, but rather the matroid
structure is disclosed row by row and the algorithm is required to
arrange the rows immediately when they become available.  The previous
section shows that this generalisation fails for odd rank. But in fact,
it fails for {\em any} $n\geq 3$, for the following simple reason. First
consider the {\em uniform} matroid $M$ of rank $n$ on $(n-1)n$ elements,
arranged in an $(n-1) \times n$-array $(e_{ij})_{ij}$, whose rows are
of course bases. Since this matroid is preserved by the entire group
$(S_n)^{n-1}$ of row permutations, we may assume that the algorithm leaves
each $e_{ij}$ in place. Then one can verify that $M$ has a one-element
extension $M \cup \{e\}$ in which the only dependent sets of size $n$
are $\{e_{ij} \mid i=1,\ldots,n-1\} \cup \{e\}$ for $j=1,\ldots,n$; here
we use that $n>2$. If we feed the algorithm an $n$-th basis containing
an $e$ such that $M \cup \{e\}$ has this structure, then it is forced
to make an error. These counterexamples can, in fact, be realised
as linear matroids over suitable fields. The difference with $\OBC_n$
is that the algorithm gets only the matroid as input, not its realisation.

So $\OBC_n$ does not have a meaningful direct generalisation to $n \times
n$-arrays. However, we think that the following problem (which asks for
an online variant of \cite{Geelen07}) may well be tractable.

\begin{prb}
Given $n$, what is the maximal value of $k$ such that there exists an
online algorithm that on sequential input of $k$ bases of a rank-$n$
matroid arranges each basis immediately as the next row in a rectangular
table, in such a way that in the resulting $k \times n$-table the columns
are independent sets?
\end{prb}

Third, a generalisation of $\RBC_n$ that we learned from a referee
concerns an $n \times n$-array of vectors in a vector space of
dimension $d \geq n$. It states that if each row consists of $n$
independent vectors, then the vectors can be permuted within their rows
such that afterwards each column consists of $n$ independent vectors,
as well. The Latin square analogue of this generalisation was proved
in \cite[Theorem 1]{Chow95}, just before Galvin proved the stronger
Dinitz conjecture \cite{Galvin94}.  For $d=n$ this generalisation is
just $\RBC_n$. Conversely, if the field is infinite (or sufficiently
large), then any instance of the generalisation can be transformed into
an instance of $\RBC_n$ by choosing a projection into $n$-space that keeps
the rows independent. So the generalisation is equivalent to $\RBC_n$ and
follows from $\ATLSC_n$, again for sufficiently large fields. However,
the reduction via a projection does not settle the {\em online} status
of this generalisation, since to choose the projection one needs to
know at least the spans of each of the $n$ rows beforehand. So a natural
question is whether, nevertheless, $\ATLSC_n$ implies the existence of
an online algorithm for this generalisation.


\end{document}